\documentclass[11pt]{article}
\usepackage{amsfonts}
\usepackage{amsmath}
\usepackage{amsthm}
\usepackage{setspace}
\usepackage[english]{babel}
\usepackage{amssymb}
\usepackage{bbm}
\usepackage{titlesec}
\usepackage{fancyhdr}
\usepackage{tocloft}
\usepackage{sectsty}
\usepackage{authblk}

\newcommand{\Z}{\mathbb{Z}}

\newcommand{\R}{\mathbb{R}}
\newcommand{\C}{\mathbb{C}}

\newcommand{\Q}{\mathbb{Q}}

\newtheorem{thm}{Theorem}[section]

\newtheorem{theorem}{Theorem}[section]
\newtheorem{proposition}[theorem]{Proposition}
\newtheorem{lemma}[theorem]{Lemma}

\theoremstyle{definition}
\newtheorem{definition}[theorem]{Definition}
\newtheorem{remark}[theorem]{Remark}

\newcommand{\mf}{\mathfrak}

\newcommand{\sbst}{\subseteq}

\newcommand{\ad}{\mathrm{ad}}

\title{On the Monodromy of Meromorphic Cyclic Opers on the Riemann Sphere}
\author{Charles LeBarron Alley} 
\date{}

\begin{document}
\maketitle

\begin{abstract}
We study the monodromy of meromorphic cyclic $\mathrm{SL}(n,\C)$-opers on the Riemann sphere with a single pole. We prove that the monodromy map, sending such an oper to its Stokes data, is an immersion in the case where the order of the pole is a  multiple of $n$.  To do this, we develop a method based on the work of M.~Jimbo, T.~Miwa, and K.~Ueno from the theory of isomonodromic deformations.  Specifically, we introduce a system of equations that is equivalent to the isomonodromy equations of Jimbo-Miwa-Ueno, but which is adapted to the decomposition of the Lie algebra $\mathfrak{sl}(n,\mathbb{C})$ as a direct sum of irreducible representations of $\mathfrak{sl}(2,\mathbb{C})$.  Using properties of some structure constants for $\mathfrak{sl}(n,\mathbb{C})$ to analyze this system of equations, we show that deformations of certain families of cyclic $\mathrm{SL}(n,\mathbb{C})$-opers on the Riemann sphere with a single pole are never infinitesimally isomonodromic.
\end{abstract}

\section{Introduction}

In this paper, we study the monodromy of meromorphic cyclic $\mathrm{SL}(n,\C)$-opers over the Riemann sphere, which are flat meromorphic connections on a trivial vector bundle generalizing the linear ordinary differential equation
\begin{equation} \label{schwarzian eq}
y^{(n)} - \phi y=0
\end{equation}
where $\phi$ is a meromorphic function.   We will focus on the case where $\phi$ is a polynomial of positive degree. In this case, \eqref{schwarzian eq} has an irregular singularity at infinity. The monodromy of this equation is defined in terms of asymptotic expansions of local solutions near infinity.  For $n\geq 2$, the monodromy at infinity is described by the \textit{Stokes phenomenon}, which refers to the behavior of certain fundamental solutions of a differential equation with an irregular singularity after analytic continuation.  This phenomenon was first noticed by G.G. Stokes and later formalized by G.D. Birkhoff (see \cite{wasow}, section 15). 

The Stokes phenomenon, which we review in Section 2, plays an important role in the irregular Riemann-Hilbert correspondence, sometimes also called the Riemann-Hilbert-Birkhoff correspondence (see \cite{boalch_poisson} for a nice survey on this topic more general than that given here). The irregular Riemann-Hilbert correspondence gives a map, called the monodromy map or Riemann-Hilbert map.  This map assigns to a flat irregular connection on a holomorphic vector bundle over a punctured Riemann surface a tuple of change of basis matrices, called  the Stokes data of the connection, relating certain horizontal local trivializations.  Two gauge equivalent connections have the same Stokes data.   

In this paper we work with irregular connections on a trivial vector bundle over the Riemann sphere with a single pole of fixed order.  We add a marking, defined in Section 3, to such connections and denote by $\mathcal{M}_{dR}$ the space of  gauge equivalence classes of marked  connections, the deRham space.
The space of Stokes data  is
$$
\mathcal{B} = \{(S_1, \dots , S_{2k+2}) \in (U_+\times U_-)^{k+1} | S_1 S_2 \cdots S_{2k+2} \in T\}
$$
where $U_+$ and $U_-$ are the subgroups of $\mathrm{SL}(n,\C)$ of upper, respectively lower, triangular unipotent matrices, $k$ is a positive integer, and $T$ is the subgroup of diagonal matrices.  This space can be viewed as the representation variety of a wild surface group for the irregular curve associated to an irregular connection on the sphere.  The geometric invariant theory quotient
$$
\mathcal{M}_B=\mathcal{B}//T,
$$
sometimes called the Betti moduli space, is then a wild character variety of the type considered in \cite{wild_boalch}.  Wild character varieties are the analogue of the ``tame" character varieties considered by C.~Simpson and others (see for example \cite{simpson}).  In this article, we work only with the space $\mathcal{B}$ as a complex manifold.

Denote the monodromy map, sending an irregular connection to its Stokes data, by $\nu$.  We give the precise definition in Section 3 and note that $\nu$ descends to a map on the quotient of $\mathcal{M}_{dR}$ by automorphisms of the bundle over automorphisms of the base fixing the pole.

In Section 4, we restrict our attention to meromorphic cyclic $\mathrm{SL}(n,\C)$-opers on the Riemann sphere with a single pole. These are connections which correspond to matrix  differential equations of the form
$$
\frac{d}{dz}Y = AY
$$ where  the matrix function $A(z)$ has the form
$$
A = \left(\begin{matrix}
0 & 1 & 0 & & \dots & & 0 \\
 & 0 & 1 & 0 &  & &  \\
\vdots &  & & \ddots & & & \vdots \\
 &  & &  & 0 &  1 & 0 \\
0 &  & &  & & 0 & 1 \\
p & 0 & & \dots & & & 0
\end{matrix}\right)
$$
where $p$ is a polynomial.  Let $d=\deg p$ and assume $d \geq 2$.  We consider the set of opers on the Riemann sphere arising from such equations.  The  group $\mathrm{Aut}(\C) = \{z \mapsto az+b\}$ of affine transformations of the plane acts on this set and  under this action, there is a distinguished normalization of such opers.  We  define $\mathcal{P}_{n,d}$ to be the set of such normalized opers.  This is an affine space modeled on the vector space of polynomials of degree $d-2$.  

We describe, in Section 4, how to define the Stokes data of an element of $\mathcal{P}_{n,d}$ and having done so define the monodromy map for mermorphic cyclic opers on the Riemann sphere with a single pole, which we also denote by $\nu$.  The main result is then:

\begin{thm}  \label{immersion}
If $d=kn$ for some positive integer $k$, then the monodromy map 
$$
\nu: \mathcal{P}_{n,d} \to \mathcal{B}
$$
is a holomorphic immersion. 
\end{thm}

Our strategy of proof is as follows. In \cite{jimbo}, M.~Jimbo, T.~Miwa, and K.~Ueno give necessary and sufficient criteria that the Stokes data of a given differential equation stay constant as the equation is deformed.  Explicitly, they consider a family of differential equations on the complex plane
$$
\frac{\partial}{\partial z}Y(z,t) = A(z,t)Y(z,t)
$$
where $A(z,t)$ is a rational matrix valued function in the variable $z$, varying in a parameter $t$.  They show that the Stokes data of this system is constant in $t$ if and only if there exists a matrix valued function $\Omega(z,t)$, rational in $z$ with the same poles as $A(z,t)$, satisfying the differential equation
\begin{equation} \label{intro jmu}
\frac{\partial}{\partial z}\Omega = \frac{\partial}{\partial t}A + [A,\Omega].
\end{equation}
Here $[\cdot,\cdot]$ is the Lie bracket, or matrix commutator.  Using this, we show that a tangent vector $\dot{A}$ is in the kernel of $d\nu$ at a point in $\mathcal{P}_{n,d}$ if and only if there exists a polynomial matrix valued function $\Omega: \C \to \mathrm{End}(\C^n)$ such that
\begin{equation} \label{JMU}
\frac{\partial}{\partial z}\Omega = \dot{A} + [A,\Omega].
\end{equation}  
This result is given in Section 4.  In Section 6, we show that \eqref{JMU} reduces to a system of $n-1$ ordinary differential equations for the coefficients of $\Omega$ using the structure of $\mathfrak{sl}(n,\C)$ as an $\mathfrak{sl}(2,\C)$-module, which is described in Section 5.  We then argue that this system can have no non-trivial polynomial solutions using degree considerations.  

The reader should note that equation \eqref{schwarzian eq} has been extensively studied in the case $n=2$ and, in that case, is known as the \textit{Schwarzian equation}.  In fact, for $n=2$, Theorem \ref{immersion} follows from a more general theorem due to I.~Bakken.  In \cite{bakken}, Bakken proves that the map $\nu$ is an immersion without any restrictions on the degree of $p$.  In that paper, the Stokes data are given by tuples of asymptotic values.  
Recently, in \cite{wild_boalch}, P.~Boalch made explicit the interpretation of Bakken's theorem in terms of Stokes data, showing that the space of asymptotic values considered by Bakken is (explicitly) algebraically isomorphic to the Betti moduli space $\mathcal{M}_B$ defined above.

Bakken was a student of Y.~Sibuya, who contributed extensively to the study of the Stokes phenomenon.  In particular, the book \cite{sibuya} is dedicated to the study of the monodromy of equation \eqref{schwarzian eq} for $n=2$ and $\phi$ a polynomial.  The work of Sibuya and Bakken was motivated in part by the work of  R.~Nevanlinna, especially the paper \cite{nevanlinna}, on functions with polynomial Schwarzian derivative.   Nevanlinna proved that a function with polynomial Schwarzian derivative of degree $d$ has exactly $d+2$ asymptotic values, which are pairwise distinct.  This gives a map from polynomials of degree $d$ to $(d+2)$-tuples of extended complex numbers which are pairwise distinct.  As a PhD research project, the author of the present article was tasked with investigating the properties of this map which, it turns out, is just a special case of the monodromy map investigated here.  

\textit{Funding}:  This work was partially supported by U.S. National Science Foundation grants DMS 1107452,
1107263, 1107367 ``RNMS: Geometric Structures and Representation Varieties (the GEAR Network)."

\textit{Acknowledgements}:  This paper is a revised version of the author's PhD thesis.  The author would like to thank his advisor David Dumas for his guidance as well as Irina Nenciu, Steven Rayan, Julius Ross, and Laura Schaposnik for their support and comments.  The author would like to thank Philip Boalch and also Davide Masoero for clarifying remarks and for introducing the author to several papers important to the present article.  Also, thanks go to Fr\'{e}d\'{e}ric Paulin for hosting the author in Orsay during a very enlightening week in March 2017 in which many fruitful conversations were had which ultimately led to this paper.  Finally, the author would like to thank the anonymous referees for their detailed comments which greatly improved this article.

\section{The Stokes Phenomenon} \label{background}

In this section we review the Stokes phenomenon in order to define the space of Stokes data.  Most of the notation and terminology given here is taken from \cite{boalch_thesis}.

A meromorphic connection on a rank $n$ vector bundle $V$ on a Riemann surface is defined by a choice of effective divisor $D$, prescribing the position and order of poles, as a $\C$-linear map
$$
\nabla:V \to V \otimes K(D)
$$
satisfying the Leibniz rule:
$$
\nabla(fs) = df\otimes s + f\nabla s.
$$
Here, we identify $V$ with its sheaf of local sections and $K(D)$ is the sheaf of meromorphic 1-forms with poles along $D$.  Given a local coordinate $z$ and a frame for $V$ we can write
\begin{equation} \label{local form}
\nabla = d - A(z)dz
\end{equation}
where $A$ is a $\mathfrak{gl}(n,\C)$ valued meromorphic function.  The matrix of 1-forms $A(z)dz$ is called a \textit{local connection form}.  A matrix valued function $Y(z)$ satisfying the linear ordinary differential equation
\begin{equation} \label{ode}
\frac{d}{dz}Y - A(z)Y = 0
\end{equation}
and whose columns are linearly independent is called a \textit{fundamental solution} of  \eqref{ode} or, equivalently, a \textit{horizontal local trivialization} of $\nabla$.

We now specialize to the case of a meromorphic connection on a trivial vector bundle over the Riemann sphere.  If we choose a coordinate $z$ so that a pole corresponds to the point at infinity  we can then write $A(z)$ as
\begin{equation} \label{Laurent}
A(z) = z^{k}\sum_{j=0}^\infty A_j z^{-j}
\end{equation}
outside of a $z$-disk of some radius.  If $A_0$ is not nilpotent and $k \geq 0$, both of these conditions being independent of local coordinate, then \eqref{ode} has an \textit{irregular singularity of Poincar\'e rank $k+1$} at the pole.  A meromorphic connection with an irregular singularity is called an \textit{irregular connection}.

In the special case where $A_0$ has distinct eigenvalues, there is an algebraic procedure to produce a unique formal solution to \eqref{ode} (see for example \cite{wasow}, Section 11, or \cite{boalch_thesis}, Appendix B).  Then we have a theorem, attributed to G.D. Birkhoff, giving the existence of holomorphic fundamental solutions in sectors based at $z=0$ which have asymptotic series representation given by this formal solution.  Before stating the theorem we must give the following definition.

\begin{definition} \cite{wasow}
Let $f(z)$ be a complex valued function defined on a set $S \sbst \C$ with infinity as an accumulation point.  Let 
$$
\hat{f} = \sum_{j=0}^\infty c_jz^{-j} \in \C[[z^{-1}]]
$$
be a formal power series in the variable $z^{-1}$.  We write $f \sim \hat{f}$ or say $f$ \textit{has asymptotic series representation} $\hat{f}$ \textit{as} $z \to \infty$ \textit{in} $S$ if for all $m \geq 0$ we have
$$
z^{m}\left(f(z) - \sum_{j=0}^m c_jz^{-j} \right) \to 0
$$
as $z \to \infty$ in $S$.  
\end{definition}

Note that $\hat{f} \in \C[[z^{-1}]]$, the ring of \textit{formal} power series in $z^{-1}$.  In this paper, a ``hat'' will indicate that a symbol is such a formal series.  Writing $Y \sim \widehat{Z}$ for $Y$ a matrix valued function and $\widehat{Z}  \in \mathrm{GL}(n,\C[[z^{-1}]])$ means that the entries of $Y$ have asymptotic series representation given by the respective entries of $\widehat{Z}$, each of which is a formal power series.  We denote by $d\widehat{Z}/dz$ the series obtained from $\widehat{Z}$ via term by term differentiation.

\begin{theorem}\cite{wasow} \label{birkhoff}
Assume, with notation as above, that $A_0$ has distinct eigenvalues $\lambda_1, \dots \lambda_n$ and $k \geq 0$.  There exists a formal matrix 
$$
\widehat{Y} \in \mathrm{GL}(n,\C[[z^{-1}]]),
$$
a diagonal scalar matrix $\Lambda$ and a diagonal matrix valued polynomial function $Q$ in the variable $z$ of degree $k+1$ with no constant term and with most singular term $\frac{z^{k+1}}{k+1}\mathrm{diag}(\lambda_1, \dots , \lambda_n)$ such that 
$$
\frac{d}{dz}\widehat{Y} = A\widehat{Y} - \widehat{Y}\left(\frac{d}{dz}Q + \Lambda z^{-1}\right).
$$

Moreover, let $S$ be an open sector based at $z=0$ of interior angle less than or equal to $\pi/(k+1)$. Then there exists  a fundamental solution $Y$ to \eqref{ode}, holomorphic on $S$, satisfying 
\begin{equation} \label{formal solution}
Y \sim \widehat{Y}z^{\Lambda} \exp(Q) \text{ as } z \to \infty \text{ in } S.
\end{equation}
\end{theorem}

For details see \cite{wasow}, Chapter 4.  
The right hand side of \eqref{formal solution} should be thought of as a power series representing the product of the formal series $\widehat{Y} \in \mathrm{GL}(n,\C[[z^{-1}]])$ with the function $z^{\Lambda} \exp(Q)$ which is holomorphic on a slit plane (corresponding to a choice of logarithm). 

The key point is that the formal solution $\widehat{Y}z^{\Lambda} \exp(Q)$ is independent of the choice of sector, thus providing a canonical way of describing solutions in a neighborhood of the pole.  In more modern terminology, what this theorem says is that a meromorphic connection $\nabla$ with local form given by \eqref{local form} where $A_0$ has distinct eigenvalues is \textit{formally gauge equivalent} in a neighborhood of an irregular singularity to one with connection form
\begin{equation} \label{irrtype}
A^0(z)dz=dQ+\Lambda z^{-1} \, dz.
\end{equation}

Explicitly, we define an action of the group $\mathrm{GL}(n,\C[[z^{-1}]])$ on the set of meromorphic connections on trivial bundles over the Riemann sphere as follows.  Given a meromorphic connection $\nabla = d - A(z)dz$ and a formal transformation $\widehat{F} \in \mathrm{GL}(n,\C[[z^{-1}]])$, define the action by $(\widehat{F},\nabla) \mapsto \widehat{F}[\nabla] =d - \widehat{F}[A]dz$ where
$$
\widehat{F}[A] = \left(\frac{d\widehat{F}}{dz} \widehat{F}^{-1} + \widehat{F}A\widehat{F}^{-1}\right).
$$
Following \cite{boalch_thesis}, we call $\widehat{F}$  a \textit{formal gauge transformation} and we say that two connections $d-A_1(z)dz$ and $d-A_2(z)dz$ are \textit{formally gauge equivalent} if there exists an $\widehat{F}$ such that $\widehat{F}[A_1]=A_2$.

Now, it is straightforward to check that if $\widehat{Y}$ is the formal series defined in Theorem \ref{birkhoff}, then $\widehat{Y}[A^0] = A$, where $A^0$ is as defined in \eqref{irrtype}.
We call  $Q$ the \textit{irregular type}, and $\Lambda$ the \textit{exponent of formal monodromy} of the given connection.  These data are all local; they depend on the pole position and order.

\begin{definition}
We call a connection to which Theorem \ref{birkhoff} applies a  \textit{semisimple irregular connection}.   
\end{definition}

Thus, a semisimple irregular connection $\nabla$ on a trivial vector bundle $V$  over the Riemann sphere is one such that the matrix $A_0$ in \eqref{Laurent} has distinct eigenvalues.  We remind the reader that to obtain the expression \eqref{Laurent}, we must first choose a local coordinate $z$ and framing for V. It is important to note that the formal gauge transformation $\widehat{Y}$ described in Theorem \ref{birkhoff} depends on the choice of frame for the vector bundle $V$.  To define this transformation, we must first diagonalize the matrix $A_0$ appearing in \eqref{Laurent}.  Choosing  $f \in \mathrm{GL}(n,\mathbb{C})$ such that $f^{-1}A_0f$ is diagonal and applying the formal procedure alluded to in Theorem \ref{birkhoff} produces a formal gauge transformation $\widehat{Y} \in \mathrm{GL}(n,\mathbb{C}[[z^{-1}]]$ with scalar term $\widehat{Y}(\infty) = f^{-1}$.  Again, the reader can consult Section 11 of \cite{wasow} or Appendix B of \cite{boalch_thesis} for details.  This will be important in the definition of the deRham moduli space given in Section \ref{jimbo section}.

\begin{definition} \cite{boalch_thesis}
A {\em compatible framing for $\nabla$ at $\infty$}  is a choice of frame for the vector bundle $V$ so that the matrix $A_0$ appearing in \eqref{Laurent} is diagonal.  
\end{definition}

Theorem \ref{birkhoff} gives us a detailed understanding of how solutions to \eqref{ode} are forced to change as they are analytically continued along paths near irregular singularities.  This is called the Stokes phenomenon and it gives rise to the notion of monodromy studied in this paper which we now describe, closely following \cite{boalch_thesis}.  From Theorem \ref{birkhoff}, we have $Q = \frac{1}{k+1}\mathrm{diag}(q_1,\dots, q_n)$ where each $q_i$ is a polynomial in $z$ of degree $k+1$ with no constant term.  Write $q_i = \lambda_iz^{k+1} + \dots$.  

\begin{definition} \cite{boalch_thesis} \label{antiStokesdef}
An \textit{anti-Stokes direction} for the system \eqref{ode} is a $d \in S^1$ such that for all $z \in \C$ with $\arg(z)=d$ and for some $i \neq j$ we have
\begin{equation} \label{antiStokes}
(\lambda_i-\lambda_j)z^{k+1} \in \R_{<0}.
\end{equation}
\end{definition}

Note that the set of all anti-Stokes directions is invariant under rotation by $\pi/(k+1)$.  It then follows that to determine the total number $r$ of anti-Stokes directions we need only consider a sector of internal angle $\pi/(k+1)$, in which there are at most $\binom{n}{2}=n(n-1)/2$ anti-Stokes rays.  Also, note that $r$ is divisible by $2(k+1)$. 

We now wish to order the anti-Stokes directions so that we can describe sectors where a canonical choice of solution to \eqref{ode} can be made.  To do this, we choose a small sector based at the origin which contains no  anti-Stokes directions.  Then, consider a  circular path about the origin, oriented counterclockwise, based at a point within the sector.  As we follow the path, we encounter a first anti-Stokes direction $d_1$.  Continuing to follow the path we eventually meet every anti-Stokes direction and label each as it is crossed until we have ordered all the anti-Stokes directions as $d_1,\dots, d_r$.  Thus, if we were to continue this procedure, we must have $d_{r+1} = d_1$ and so the index of the anti-Stokes directions will be taken modulo $r$.

Define the $i^{th}$ \textit{Stokes sector} to be 
$$
\mathrm{Sect}_i = \mathrm{Sect}(d_{i},d_{i+1}) = \{z \in \C | d_{i} < \arg(z) < d_{i+1}\}. $$
Then define the $i^{th}$ \textit{supersector} to be
$$
\widehat{\mathrm{Sect}_i} = \mathrm{Sect}\left(d_{i} - \frac{\pi}{2(k+1)}, d_{i+1}+\frac{\pi}{2(k+1)}\right)
$$ 
Then we have the following result. 

\begin{proposition} \cite{boalch_thesis}
In each $\mathrm{Sect}_i$ there is a unique choice of invertible holomorphic fundamental solution $\Phi_i$ of \eqref{ode} which, upon analytic continuation to $\widehat{\mathrm{Sect}_i}$, has asymptotic series representation as $z \to \infty$ in $\widehat{\mathrm{Sect}_i}$ given by the formal solution of Theorem \ref{birkhoff}; that is as $z$ goes to $\infty$ in $\widehat{\mathrm{Sect}_i}$ we have
$$
\Phi_i(z) \sim \widehat{Y}z^{\Lambda} \exp(Q).
$$
\end{proposition}

\begin{definition} \cite{boalch_thesis}
We call $\Phi_i$ the \textit{canonical fundamental solution} of \eqref{ode} on Sect$_i$. Note that this depends on a labeling of the anti-Stokes directions and a choice of $\log z$.  
\end{definition}

The number of canonical fundamental solutions is equal to the number of anti-Stokes directions $r$; there is one associated to each supersector.  As above, the index $i$ is to be taken mod $r$.  Thus, in particular we have $\Phi_0 = \Phi_r$.

Next, $\Phi_i$ and $\Phi_{i+1}$ are both fundamental solutions of \eqref{ode} which extend to fundamental solutions on the intersection $\widehat{\mathrm{Sect}_i} \cap \widehat{\mathrm{Sect}_{i+1}}$.  With this, for $1 \leq i \leq r-1$ define
\begin{equation} \label{Stokesfactor}
K_{i+1} = \Phi_{i+1}^{-1}\Phi_{i}
\end{equation}
and
$$
K_1 = \Phi_1^{-1}\Phi_r\exp(-2\pi i \Lambda).
$$
We call $K_i$ the $i^{th}$ \textit{Stokes factor}.
See for example \cite{boalch_thesis} (of course), but also \cite{wasow} section 15, or \cite{Lutz}. 

Next, let $U_+,U_-$ be the upper, respectively, lower triangular unipotent subgroups of $\mathrm{SL}(n,\C)$.  Then we have:

\begin{proposition} \cite{boalch_thesis} \label{Uplusminus}
Choose a labeling of anti-Stokes directions as above and write $r = (2k+2)\ell$ for some positive integer $\ell$.  Then there is a unique permutation matrix $P \in \mathrm{GL}(n,\C)$ such that for $i \geq 1$, the multiplication map
$$
(K_{i\ell}, \dots , K_{(i-1)\ell+1}) \mapsto P^{-1}K_{i\ell}\cdots K_{(i-1)\ell+1}P
$$
is a diffeomorphism onto $U_+$ or $U_-$ depending on whether $i$ is odd or even, respectively.
\end{proposition}

\begin{definition} \cite{boalch_thesis}
The $i^{th}$ \textit{Stokes matrix} is the $\ell$-fold product of Stokes factors appearing in Proposition \ref{Uplusminus}:
$$
S_i = K_{i\ell}\cdots K_{(i-1)\ell+1}.
$$
\end{definition}

\begin{lemma}\cite{boalch_thesis} \label{stokesphenomenon}
For $1 \leq i \leq 2(k+1)$ the fundamental solution $\Phi_{i\ell}$ can be analytically continued to Sect$_{(i+1)\ell}$ and in that sector we have
$$
\Phi_{i\ell}= \Phi_{(i+1)\ell}S_{i+1}, 
$$
unless $i = 2(k+1)$ in which case we have $\Phi_{(2k+2)\ell} = \Phi_{\ell}S_1\exp(2\pi i \Lambda)$.  Moreover, the monodromy of the system \eqref{ode} around a simple closed loop about $z=\infty$ is conjugate to the product
$$
S_{2k+2}S_{2k+1}\cdots S_2S_1\exp(2 \pi i \Lambda).
$$
\end{lemma}

\

The behavior described in Lemma \ref{stokesphenomenon} is what is usually referred to as the \textit{Stokes phenomenon}.

In summary, given the differential equation \eqref{ode} where we assume that the matrix valued function $A(z)$ has germ at a pole with most singular term $A_0$, a diagonal matrix with distinct eigenvalues, we have the associated monodromy data consisting of tuples of Stokes matrices $(S_1,S_2,\dots, S_{2k+2})$, where $P^{-1}S_iP \in U_{\pm}$ for some permutation matrix $P$, and the exponent of formal monodromy $\Lambda$. In the special case of a semi-simple irregular connection on a vector bundle on the Riemann sphere with only a single pole, the Stokes matrices along with the exponent of formal monodromy are the only monodromy data we need to consider\footnote{One should be careful as there is little consistency in terminology across the literature.  The terms Stokes matrix, Stokes factor, Stokes multiplier, etc.~are all used in different contexts and often refer to similar but different constructions.  We have chosen to follow \cite{boalch_thesis}, as it gives a very comprehensive and modern treatment of the many perspectives on the Stokes phenomenon.  In the case of multiple poles, we have the Stokes matrices and exponent of formal monodromy, as defined above, at each pole plus a set of connection matrices relating fundamental solutions at different poles.  }.
 As the monodromy about a contractible loop must be equal to the identity, we obtain the following restriction on the data which we consider:
\begin{equation} \label{monodromyidentity}
S_{2(k+1)} \cdots \, S_1\exp(2 \pi i \, \Lambda) = I_n.
\end{equation}
Also, it follows from the residue theorem for Riemann surfaces that 
\begin{equation} \label{trace0}
0= \mathrm{Tr}(\Lambda).
\end{equation}

\section{The Monodromy Map and Isomonodromic Deformations}
\label{jimbo section}

In this section we restrict our attention to irregular connections on a trivial vector bundle over the Riemann sphere, which we identify with the complex projective line $\C\mathbb{P}^1$.  
We fix a trivial rank $n$ vector bundle $V$ on $\C\mathbb{P}^1$ and an integer $k \geq 0$.  We define a \textit{marked triple} as a tuple $(\nabla,f,v)$ consisting of a semisimple irregular connection $\nabla$ on $V$ with a single pole at $\infty$, which we assume is an irregular singularity of Poincar\'e rank $k+1$, a compatible framing $f$ of $\nabla$ at $\infty$, and a non-zero vector $v \in T_{\infty}\C\mathbb{P}^1$ which is not an anti-Stokes direction\footnote{Previously we defined an anti-Stokes direction as an element of the circle $S^1$.  Here, when we say that a non-zero tangent vector $v$ is not an anti-Stokes direction, what we mean is that the associated element $v/|v|$ of $S^1$ is not an anti-Stokes direction.}. With the choice of $v$ we can order the Stokes sectors, as above, so that $v$ is interior to one of the Stokes sectors and subsequent sectors are met by counterclockwise rotation about the pole.   This induces an ordering on the Stokes matrices.  By the conventions of Section 2, $v$ will belong to the last Stokes sector $\mathrm{Sect}_r=\mathrm{Sect}_0$.

Define the \textit{deRham moduli space} $\mathcal{M}_{dR}$ as the set of equivalence classes of marked triples under the following equivalence relation. First, if $v$ and $v'$ are interior to the same Stokes sector then we identify the triples $(\nabla,f, v)$ and $(\nabla,f, v')$.  Second, if $g$ is a gauge transformation for $V$, i.e. a fiber preserving automorphism of $V$, or equivalently, an automorphism of $V$ over the identity on $\C\mathbb{P}^1$, then $g$ acts on a marked triple by
$$
g.(\nabla,f,v) = (g[\nabla],gf,v)
$$
where
$g[\nabla]s = g\nabla(g^{-1}s)$ for a section $s$ of $V$.   $\mathcal{M}_{dR}$ is the set of equivalence classes of marked triples under this action by holomorphic $g$.

A point $[(\nabla,f,v)] \in \mathcal{M}_{dR}$ determines a tuple of monodromy data $$(S_1,\dots,S_{2k+2}, \Lambda)$$ satisfying \eqref{monodromyidentity} and \eqref{trace0} where the ordering of the Stokes matrices is determined by the choice of $v$.  For ease of notation, we will sometimes write simply $\nabla$ for a point in $\mathcal{M}_{dR}$, identifying an equivalence class with a chosen representative and  suppressing the framing $f$ and the vector $v$.

Now, we define the space of Stokes data as 
$$
\mathcal{B} = \left\{(S_1, \dots, S_{2k+2}) \, \Big| \, S_{2k+2} \cdots S_1 \in T\right\}
$$
where $T \subset \mathrm{SL}(n,\C)$ is the set of $n \times n$ diagonal matrices of determinant 1.
Then we have a map, called the \textit{monodromy map},
$$
\nu: \mathcal{M}_{dR} \to \mathcal{B},
$$
which is a holomorphic map of complex manifolds taking a connection to its Stokes data.

It is important to note at this point that the monodromy map descends to a map on the quotient of $\mathcal{M}_{dR}$ by the group $\mathrm{Aut}(V,\infty)$ of automorphisms of $V$ over automorphisms of $\C\mathbb{P}^1$ which fix $\infty$, acting on $\mathcal{M}_{dR}$ as follows.  If $\Psi$ is an automorphism of $V$ over $\psi$, where $\psi \in \mathrm{Aut}(\C\mathbb{P}^1)$ satisfies $\psi(\infty)=\infty$, then we define
$$
\Psi.[(\nabla,f,v)]=[(\Psi[\nabla],\Psi f, d\psi_{\infty}v)],
$$
where $\Psi[\nabla]s = \Psi\nabla(\Psi^{-1}s)$ for a section $s$ of $V$.  If we let
$$
\mathcal{M}_{dR}' = \mathcal{M}_{dR}/\mathrm{Aut}(V,\infty),
$$
then we can define
$$
\nu: \mathcal{M}_{dR}' \to \mathcal{B}
$$
exactly as above.  To see that this makes sense, note that under the action of $\Psi$  the canonical solutions $\Phi_i$ for $\nabla$ are mapped to $\Psi\Phi_i$ and so it follows from \eqref{Stokesfactor} that the Stokes matrices are left invariant.  Moreover, $d\psi_{\infty}$ simply rotates the anti-Stokes directions, leaving the ordering of the Stokes matrices invariant. Thus, the map $\nu$ descends to a well defined map on $\mathcal{M}_{dR}'$.
We will return to this point in section 4.

Furthermore, if we let $X$ be the space of irregular types of semisimple irregular connections on $V$ with an irregular singularity of Poincar\'e  rank $k+1$ at $\infty$, then $\mathcal{M}_{dR}$ has the structure of a flat fiber bundle over $X$.  The restriction of $\nu$ to a fiber is a submersion and biholomorphism onto its image in $\mathcal{B}$ (see \cite{boalch_thesis}, Corollary 4.13). 
This fact is one of a number of similar theorems commonly known as the irregular Riemann-Hilbert correspondence (sometimes also called the Riemann-Hilbert-Birkhoff correspondence).

\begin{definition}
We call a submanifold $N \subset \mathcal{M}_{dR}$ \textit{isomonodromic} if $\nu$ restricted to $N$ is locally constant.  Or, equivalently, $N$ is isomonodromic if it is tangent to the distribution $\ker d\nu$.
\end{definition}

The goal of this paper is to prove that the monodromy map $\nu$, restricted to a particular family of irregular connections on a trivial bundle on $\C\mathbb{P}^1$ with a single pole, is an immersion.  To do this we apply a result of Jimbo, Miwa, and Ueno from \cite{jimbo} which gives a criterion that a family of irregular connections on $\C\mathbb{P}^1$ be isomonodromic.  Specifically, they prove that a family being isomonodromic is equivalent to the existence of a certain rational solution $\chi \in  H^0(\mathrm{End}(V))$ to the differential equation
\begin{equation} 
\frac{\partial}{\partial z}\chi = \frac{\partial}{\partial t}A + [A, \chi]
\end{equation}
where $A(t,z)$ is a family of rational matrix valued functions of $z$ varying holomorphically in a parameter $t \in X$, the space of irregular types.  To prove Theorem \ref{immersion}, we analyze this equation evaluated at a point (replacing $\frac{\partial}{\partial t}A$ with $\dot{A}\in T_\nabla\mathcal{M}_{dR}$ for some fixed $\nabla$).  

We take a brief aside now to discuss the notion of \textit{infinitesimally isomonodromic} families in the general setting.  Consider the trivial vector bundle $\C \times \C^n$ over the complex plane and let $\mathcal{A}$ be the space of flat connections.  We define a map $\Omega:T\mathcal{A} \to H^0(\mathrm{End}(\C^n))$ as follows.  Given $(\nabla_0,\dot{A}) \in T\mathcal{A}$, consider a smooth family of connections $\nabla_t: (-\varepsilon, \varepsilon) \to \mathcal{A}$ with velocity vector $\dot{A}$, i.e.
$$
\left.\frac{\partial}{\partial t}\nabla_t\right|_{t=0} = \dot{A}.
$$
Let $Y_t:\C \to GL(n,\C)$ be a smoothly varying family of $\nabla_t$-trivializing gauge transformations (i.e. fundamental solutions), uniquely determined for each $t$ by imposing the initial conditions $Y_t(1) = M_t$, and define
$$
\Omega(\nabla_0,\dot{A}, M_0,\dot{M}) = \dot{Y}Y_0^{-1}
$$
where $\dot{Y}= \left.\frac{\partial}{\partial t}Y_t \right|_{t=0}$ and similarly for $\dot{M}$.  Write $\nabla_0 = d + A(z)dz$.  Then $\Omega(\nabla_0, \dot{A},M_0,\dot{M})$ is the unique solution to the linear ordinary differential equation
\begin{equation} \label{inf_jmu}
\frac{\partial}{\partial z}\chi = \dot{A} + [A,\chi]
\end{equation}
with initial condition $\chi(1) = \dot{M}M_0^{-1}$.  
By construction, $\Omega(\nabla_0,\dot{A},M_0,\dot{M})$ is smooth wherever $\nabla_0$ is smooth.  Note that 
$$\Omega(\nabla_0,\dot{A},M_0,\dot{M}) = \Omega(\nabla_0,\dot{A},I_n,\dot{M}M_0^{-1}).$$   Thus, to simplify notation, write $\Omega(\nabla_0,\dot{A},M) = \Omega(\nabla_0,\dot{A},I_n,M)$ for the unique solution to \eqref{inf_jmu} with initial condition $\chi(1)=M$.

\begin{theorem}\cite{jimbo} \label{jimbo}  Let $\nabla_0 \in \mathcal{M}_{dR}$ and $\dot{A} \in T_{\nabla_0}\mathcal{M}_{dR}$ be given. Assume that the exponent of formal monodromy $\Lambda_0$ of $\nabla_0$ is constant to first order in the direction $\dot{A}$.  Then, the Stokes matrices of $\nabla_0$ at $ \infty$ are constant to first order in the direction $\dot{A}$ or, equivalently, $\dot{A} \in \ker d_{\nabla_0} \nu$ if and only if  there exists an $M$ such that $\Omega(\nabla_0,\dot{A},M)$ has only a pole at $\infty$.
\end{theorem}

\section{Monodromy of Meromorphic Cyclic Opers} \label{oper section}

\begin{definition} \label{oper definition}  An
$\mathrm{SL}(n,\C)$-\textit{oper} on a Riemann surface $X$ is a triple $(V,\mathcal{F},\nabla)$ consisting of a holomorphic bundle $V$, a filtration $\mathcal{F} = \{V_i\}$ of $V$, $0 = V_0 \subset V_1 \subset \dots \subset V_{n-1} \subset V_n=V$, and a holomorphic connection $\nabla$ inducing the trivial connection on $\det V$ such that
\begin{itemize}
\item[i)] $\nabla(V_i) \subset V_{i+1}\otimes K$
\item[ii)] for $1 \leq i \leq n-1$ there is an isomorphism $V_{i}/V_{i-1} \to (V_{i+1}/V_i) \otimes K$ induced by $\nabla$.
\end{itemize}
Given a vector bundle $V$, a connection and filtration satisfying (i) and (ii) is called an \textit{oper structure}.
\end{definition}

Given a coordinate chart $z$, an $\mathrm{SL}(n,\C)$-oper connection is gauge equivalent to a unique connection of the form:
\begin{equation} \label{oper ode}
d - \left(\begin{matrix}
0 & 1 & 0 & & & &  \\
& 0 & 1 & 0 & & & \\
& &  \ddots & & &  \\
&  & &  0 &  1 & 0 \\
&  & &   & 0 & 1 \\
Q_n & Q_{n-1} &  \dots &  & Q_2 & 0
\end{matrix}\right)dz
\end{equation}
where $Q_j$ is a holomorphic function (see \cite{ben-zvi}).  

\begin{definition}
A \textit{cyclic} $\mathrm{SL}(n,\C)$-\textit{oper} is an $\mathrm{SL}(n,\C)$-oper that in any local coordinate $z$ corresponds to a connection of the form \eqref{oper ode}  with $Q_{n-1}= \dots = Q_2 =0$.  That is, one with connection
\begin{equation} \label{cyclic oper ode}
d - \left(\begin{matrix}
0 & 1 & 0 & & & &  \\
& 0 & 1 & 0 & & & \\
& &  \ddots & & &  \\
&  &   & 0 &  1 & 0 \\
 &   &  & & 0 & 1 \\
Q &  0  & \dots & &  & 0
\end{matrix}\right) dz. 
\end{equation}
\end{definition}
This definition first appeared in \cite{acosta} and was motivated by the definition of cyclic Higgs bundles (which we will not discuss here).  We will note however, following \cite{acosta}, that there is a bijective correspondence between $\mathrm{SL}(n,\C)$-opers and the Hitchin base
$$
\mathcal{H}_n = \bigoplus_{i=2}^n H^0(X,K^i).
$$
(See \cite{wentworth}).  Cyclic $\mathrm{SL}(n,\C)$-opers correspond to tuples of the form 
$$(0,\dots,0,\phi_n) \in \mathcal{H}_n$$ where $\phi_n = Qdz^n$ in the coordinate $z$ and $Q$ is as in \eqref{cyclic oper ode}. In particular, the function $Q$ transforms as an $n$-differential under change of coordinates.

We wish to study \textit{meromorphic} cyclic $\mathrm{SL}(n,\C)$-opers on $\C\mathbb{P}^1$. We define a meromorphic oper by replacing the bundle $K$ with $K(D)$, for some effective divisor $D$, in definition \ref{oper definition}.  It turns out that the bundle $V$ is uniquely determined up to isomorphism by the condition that it admits an oper structure and that $V = J^{n-1}(K^{\frac{1-n}{2}})$, the bundle of $n-1$ jets of sections of the bundle $K^{\frac{1-n}{2}}$.  This requires a choice of square root, or \textit{spin structure}, of the canonical bundle (see \cite{ben-zvi}).  On the Riemann sphere, there is a unique choice of spin structure, and in particular $K^{1/2} = \mathcal{O}(-1)$.  Moreover, one can show that $J^{n-1}(K^{\frac{1-n}{2}})$ is trivial on $\C\mathbb{P}^1$.

In this paper, we consider the case of meromorphic cyclic $\mathrm{SL}(n,\C)$-opers on $\C\mathbb{P}^1$ with only a single pole.  Choosing a coordinate so that the pole corresponds to the point at infinity, such an oper connection corresponds to a choice of polynomial $n$-differential.  That is, we consider $Q = p$ where $p$ is a polynomial of degree $d$.

The system of ordinary differential equations given by \eqref{cyclic oper ode} then corresponds
to the $n^{th}$ order differential equation mentioned in the introduction:
\begin{equation} \label{poly ode}
y^{(n)} - py=0.
\end{equation}
As previously noted, in the case $n=2$, \eqref{poly ode} is known as the Schwarzian equation and has been extensively studied.  If $p(z)=z$ then equation \eqref{poly ode} is the Airy equation; while in the case $p(z) = z^2 + c$, it is the Hermite-Weber equation (see \cite{wasow}).  In \cite{sibuya}, Y. Sibuya gives a comprehensive treatment of equation \eqref{poly ode} in the case $n=2$ and for $p$ an arbitrary polynomial with a particular emphasis on asymptotic analysis and the Stokes phenomenon.  

From now on, we assume that $d = kn$ for some positive integer $k$. Then the oper connection given by \eqref{cyclic oper ode}, with $Q=p$ a polynomial of degree $d$, is gauge equivalent, via the diagonal gauge transformation 
\begin{equation} \label{gauge}
g = \mathrm{diag}\left(z^{(n-j)k} \mid 0 \leq j \leq n-1\right),
\end{equation}
to one of the form 
\begin{eqnarray} 
\nabla_0 = g[\nabla] = d - z^{k}\left(\sum_{j=0}^\infty A_jz^{-j}\right)dz, \label{semisimple oper at 0} 
\end{eqnarray}
where $A_0$ has distinct eigenvalues.  That is, $\nabla_0$ is a semi-simple irregular connection on the oper bundle $V=J^{n-1}(K^{\frac{1-n}{2}})$. Note that the gauge transformation represented by formula \eqref{gauge} is meromorphic on the Riemann sphere and only depends on the degree $d$ of the polynomial $p$.  In particular, $g$ does not depend on the coefficients of $p$.  Furthermore, we see that the assumption $d=kn$ is crucial here.  For, without this assumption formula \eqref{gauge} would define a multi-valued function and formula \eqref{semisimple oper at 0} would not describe a meromorphic connection on the Riemann sphere.

We define a \textit{marked cyclic oper} to be a triple $(\nabla,f,v)$ where $\nabla$ is a cyclic oper connection and $(\nabla_0,f,v)$ is a marked triple, i.e.~$f$ is a compatible framing for $\nabla_0$ at $\infty$ and $v$ is not an anti-Stokes direction for $\nabla_0$.  We then define the Stokes data of $(\nabla,f,v)$ as the Stokes data given by $[(\nabla_0,f,v)] \in \mathcal{M}_{dR}$. 

Next, we note that we can change coordinates by an affine transformation $z \mapsto az+b$ for $a,b \in \C, a \neq 0$ so that $p$ becomes monic and trace zero, i.e. for some coefficients $c_0, c_1, \dots, c_{d-2} \in \C$ we can write
$$
p(z) = z^d + c_{d-2}z^{d-2} + \dots + c_1z + c_0.
$$
In fact, there are exactly $d+n$ choices of $a\in \C^*$ for such a transformation, differing from one another by multiplication of $a$ by $(d+n)^{th}$ roots of unity. We denote by $\mathrm{Aut}(\C)$ the group of affine transformations of the plane  and we identify this with the group of automorphisms of $\C\mathbb{P}^1$ fixing the pole at infinity.  

Applying such a transformation, the coefficient of the most singular term in \eqref{semisimple oper at 0} is
$$
A_0 = \left(
\begin{matrix}
0 & 1 & 0 &   & \dots & 0 \\
0 & 0 & 1 & 0 &  \dots & 0 \\
\vdots  &   & \ddots & &&  \\
& & & 0 & 1 & 0 \\
0 &   & &        & 0  & 1   \\
1 & 0 & & \dots &      & 0
\end{matrix}
\right),
$$
which has eigenvalues $\lambda_j = \zeta^j$ for $0 \leq j \leq n-1$ where $\zeta$ is some primitive $n^{th}$ root of unity. Using this normalization, we can compute the anti-Stokes directions for the connection $\nabla_0$.  All monic polynomials give the same Stokes sectors at infinity, which are the sectors
$$
\left\{\frac{(2j-1)}{2n(k+1)}\pi < \arg\left(z\right) < \frac{(2j+1)}{2n(k+1)}\pi \, \Big| \, j \in \Z \right\}
$$
if $n$ is odd, and
$$
\left\{\frac{j}{n(k+1)}\pi < \arg\left(z\right) < \frac{(j+1)}{n(k+1)}\pi \, \Big| \,  j \in \Z \right\}
$$
if $n$ is even.  The thing to notice is that for $n$ odd, one Stokes sector is always symmetric about the positive real axis, while if $n$ is even then the positive real axis lies along an anti-Stokes direction.  

With this observation, let us set
$$
\mathrm{Sect}_0 = \left\{-\frac{\pi}{2n(k+1)} < \arg\left(z\right) < \frac{\pi}{2n(k+1)}  \right\}
$$
if $n$ is odd, and
$$
\mathrm{Sect}_0 = \left\{0 < \arg\left(z\right) < \frac{\pi}{n(k+1)}  \right\}
$$
if $n$ is even.  Then we have $\mathrm{Sect}_j = e^{\frac{\pi \sqrt{-1} j}{n(k+1)}} \mathrm{Sect}_0$ for $0 \leq j \leq 2n(k+1)-1$.
This choice corresponds to the marking $(\nabla,f,v_0)$ where $\nabla$ is a cyclic oper connection corresponding to a monic polynomial $n$-differential of degree $d = kn$ and the direction given by $v_0$ is $e^{\frac{\pi \sqrt{-1}}{4n(k+1)}}$, for example. We call this choice of labeling of the Stokes sectors the \textit{canonical normalization} of the oper connection $\nabla$.

Having specified a marked cyclic oper $(\nabla,f, v)$, there is only one value of $a \in \C^*$ so that the action of the transformation $z \mapsto az+b$ on $(\nabla,f, v)$ is equivalent to $(\nabla,f, v_0)$ (in the sense that $av$ and $v_0$ both belong to $\mathrm{Sect}_0$).  Thus, given a meromorphic cyclic oper corresponding to an $n$-differential $pdz^n$ where, $p$ is a polynomial of degree $d = kn$, there is a unique element of the $\mathrm{Aut}(\C)$ orbit which gives a canonical normalization.  This gives a bijection between the $\mathrm{Aut}(\C)$ equivalence classes of marked cyclic opers $(\nabla, f, v)$, where $\nabla$ corresponds to a monic, trace zero polynomial $n$-differential of degree $d = kn$, and the set of canonically normalized triples $(\nabla, f,v_0)$, where $v_0$ is chosen as above.  Fix a matrix $f_0 \in \mathrm{GL}(n,\mathbb{C})$ diagonalizing $A_0$ (this choice induces a canonical compatible framing for the connection $\nabla_0$) and denote by $\mathcal{P}_{n,d}$ the set of canonically normalized triples $(\nabla, f_0, v_0)$.  Then  we have a map, which we will refer to as the monodromy map for cyclic opers and denote by the same symbol $\nu$ as the monodromy map for semi-simple irregular connections,
$$
\nu: \mathcal{P}_{n,d} \to \mathcal{B}
$$
defined by $\nu(\nabla,f_0,v_0) = \nu([\nabla_0,f_0,v_0])$, where the right hand side is as defined in Section \ref{jimbo section}. 

Recall, that $\nu$ descends to a map on the space $\mathcal{M}_{dR}'$ of equivalence classes of marked triples under the action of the group of bundle automorphisms over automorphisms of $\C\mathbb{P}^1$ which fix the point at infinity.  The set $\mathcal{P}_{n,d}$ is a convenient slice for the subset of $\mathcal{M}_{dR}'$ represented by cyclic opers because it has a natural affine structure modeled on the space of all polynomials of degree $d-2$.  (Recall that we have assumed that $d=kn$ and so, in particular, we have $d \geq 2$.)  We will henceforth simply denote a point in $\mathcal{P}_{n,d}$ by $\nabla$ and the corresponding Stokes data by $\nu(\nabla)$.  

Thus, fixing a connection $\nabla \in \mathcal{P}_{n,d}$, a tangent vector at $\nabla$ corresponds to a matrix $\dot{A} = \left(\begin{matrix}
&0 \\
\dot{p}&
\end{matrix}\right)$
where $\dot{p}$ is an arbitrary polynomial of degree $d-2$. Theorem \ref{immersion} is equivalent to the statement that $\dot{A} \in \ker d_{\nabla}\nu$ if and only if $\dot{A} = 0$.  The next proposition gives a necessary and sufficient criterion, similar to that of Theorem \ref{jimbo}, that $\dot{A} \in \ker d_{\nabla}\nu$.  

\begin{proposition} \label{oper jmu}
Fix $\nabla \in \mathcal{P}_{n,d}$ where $d = kn$ and consider a tangent vector $\dot{A} \in T_\nabla\mathcal{P}_{n,d}$.  Then $\dot{A} \in \ker d_{\nabla}\nu$ if and only if there exists a polynomial matrix valued function $\Omega:\C \to \mathrm{End}(\C^n)$ such that
$$
\frac{\partial}{\partial z}\Omega = \dot{A} + [A,\Omega].
$$
\end{proposition}

\begin{proof}  We wish to apply Theorem \ref{jimbo}. With notation as above write $\nabla_0 = d + B(z) dz$.  By definition, $\nu(\nabla) = \nu(\nabla_0)$ and so $\dot{A} \in \ker d_{\nabla} \nu$ if and only if $\dot{B} = g\dot{A}g^{-1} \in \ker d_{\nabla_0}\nu$.  Theorem \ref{jimbo} says that this is the case if and only if there exists a matrix $M$ such that the function $\Omega(\nabla_0,\dot{B},M)$ has only a pole at $\infty$. Recall that $\Omega(\nabla_0,\dot{B},M)$ is the unique solution to the ordinary differential equation
\begin{equation} \label{jmu proof}
\frac{\partial}{\partial z}\chi = \dot{B} + [B,\chi].
\end{equation}
satisfying the initial condition 
$$
\chi(1) = M.
$$
Recalling that $B= \frac{\partial g}{\partial z}g^{-1} + gAg^{-1}$, we find that $\chi = g^{-1}\Omega(\nabla_0,\dot{B}, M)g$ is the unique solution to the equation 
\begin{equation}
\frac{\partial}{\partial z}\chi = \dot{A} + [A,\chi] 
\end{equation}
satisfying the initial condition $\chi(1)= g(1)^{-1}Mg(1)$.  That is, we have
\begin{equation} \label{omega equation}
\Omega(\nabla, \dot{A}, g(1)^{-1}Mg(1)) = g^{-1}\Omega(\nabla_0,\dot{B},M)g.
\end{equation}
The right hand side of this equation has a pole at $\infty$; thus, so does the left hand side.  But, by construction, $\Omega(\nabla,\dot{A},g(1)^{-1}Mg(1))$ is holomorphic on $\C$.  That is, $\Omega = \Omega(\nabla,\dot{A},g(1)^{-1}Mg(1))$ is a polynomial function of $z$.
\end{proof}

\section{Representation Theory of $\mf{sl}(2,\C)$} \label{slrep}

In this section we briefly review some standard representation theory which will be put to use in the next section.  While the exposition given here is our own, the majority of this material can be found in standard texts such as \cite{fulton}.  However, the last result given in this section, Lemma \ref{ibracket}, is specific to our method.

We begin with the fact that $\mf{sl}(2,\C)$ has a unique irreducible representation of each dimension described as follows.  Let $V = \C^2$ with the standard action of $\mf{sl}(2,\C)$.  For $n \geq 2$ the $n$-dimensional irreducible representation is $V_n = \Sigma^{n-1}V$, where $\Sigma^k V$ is the $k^{th}$ symmetric tensor power of $V$.  

This induces an action of $\mf{sl}(2,\C)$ on $\C^n$ (by identifying $\C^n$ with $V_n$) and a Lie algebra homomorphism $\sigma_n: \mf{sl}(2,\C) \to \mf{sl}(n,\C)$.  Then, composing with the adjoint representation of $\mf{sl}(n,\C)$,
$$
\mf{sl}(2,\C) \xrightarrow{\sigma_n} \mf{sl}(n,\C) \xrightarrow{\mathrm{ad}} \mathrm{End}(\mf{sl}(n,\C)),
$$
we have realized $\mf{sl}(n,\C)$ as a representation of $\mf{sl}(2,\C)$ and it therefore decomposes into a direct sum of irreducible representations.  The following theorem and the construction to follow is implicitly given in \cite{kostant}.

\begin{theorem}
As an $\mf{sl}(2,\C)$-module,
$$
\mf{sl}(n,\C) \cong \bigoplus_{i=1}^{n-1}V_{2i+1}.
$$
\end{theorem}

\

We now describe a basis for $\mf{sl}(n,\C)$ adapted to this direct sum decomposition. Let $\{e,f,h\}$ be the usual basis of $\mathfrak{sl}(2,\C)$ satisfying the Serre relations. Then one can show by direct computation that
$$
\sigma_n(e) = \left(\begin{matrix} 
0 & 1 & 0 & \\
& 0 & 2 & 0 \\
& & \ddots &  &  \\
& & & 0 & n-2 & 0 \\
& & & & 0 & n-1 \\
& & & & & 0
\end{matrix}\right),
$$
$$
\sigma_n(f) =  \left(\begin{matrix}
0 &  &  & \\
n-1 & 0 &  &  \\
0 & n-2 & 0 &  &  \\
& & &  \ddots &  \\
& & 0 & 2 & 0  \\
& & & 0 & 1 & 0
\end{matrix}\right),
$$
and
$$
\sigma_n(h)=\left(\begin{matrix} 
n-1 &  &  & \\
    & n-3 &  &  \\
& &\ddots &   &  \\
& &  & -(n-3) &   \\
& & &   & -(n-1)
\end{matrix}\right).
$$

\

\begin{definition}
With notation as above, define 
$$
\tilde{e} = \sigma_n(e), \quad   \tilde{f}=\sigma_n(f), \quad \tilde{h} = \sigma_n(h).
$$
\end{definition}

Now, denote by $S(i)$ the $2i$-eigenspace of $\mathrm{ad}_{\tilde{h}}$.  Concretely, $S(i)$ is the subset of matrices whose only non-zero entries lie in the $i^{th}$ off-diagonal.  Thus, $S(0)$ is the subset of diagonal matrices.  For positive $i$, $S(i)$ is the subset of matrices whose only non-zero entries lie in the $i^{th}$ super-diagonal and for negative $i$, $S(i)$ is the subset of matrices whose only non-zero entries lie in the $i^{th}$ sub-diagonal.  

We record the next result as a lemma for its importance in what follows.  The proof is by direct computation.

\begin{lemma}
For $1 \leq i \leq n-1$, there is a unique element $f_i \in \mf{sl}(n,\C)$ which is in the intersection of $S(-i)$ with the kernel of the map $\mathrm{ad}_{\tilde{f}}$ and whose only non-zero coefficients are positive integers, the least of which is 1.  Each $f_i$ is a lowest weight vector  of weight $-2i$ for the action of $\mathfrak{sl}(2,\C)$ on $\mathfrak{sl}(n,\C)$.
\end{lemma}

For example, $f_1 = \tilde{f}$ and $f_{n-1}$ is the matrix whose only non-zero entry is a 1 in the bottom left corner, i.e. the entry in row $n$, column 1:
$$
f_{n-1} = \left(
\begin{matrix}
0 & 0 & &\dots && 0 \\
0 & 0 & &\dots && 0 \\
  &   & &\ddots && \\
0 & 0 & &\dots && 0 \\
1 & 0 & &\dots && 0
\end{matrix}\right)
$$

Now, for $1 \leq i \leq n-1$ and $-i \leq j \leq i$, define 
$$
v_{i,j} = (\mathrm{ad}_{\tilde{e}})^{i+j}f_i.
$$
Then, one observes that $v_{i,j} \in S(j)$.  Furthermore, note that $v_{i,-i} = f_i$ and $v_{i,j+1} = \mathrm{ad}_{\tilde{e}}(v_{i,j})$ while $\mathrm{ad}_{\tilde{f}}(v_{i,j})$ is a multiple of $v_{i,j-1}$ (see Lemma \ref{ftilde} below).  More generally, for integers $j$ and $\ell$ such that $-(n-1) \leq j+\ell \leq n-1$ we have $[v_{i,j},v_{k,\ell}] \in S(j+\ell)$. 

\begin{lemma}
For fixed $i$, we have an isomorphism of $\mathfrak{sl}(2,\C)$-modules
$$
V_{2i+1} \cong \mathrm{span}\{v_{i,j} \mid -i \leq j \leq i\}.
$$
\end{lemma}

\begin{proof}
The set $\{v_{i,j}\}$ span a subspace of dimension $2i+1$ of $\mf{sl}(n,\C)$ which is invariant under the action of $\mf{sl}(2,\C)$.  It is irreducible by the theory of highest weight vectors.
\end{proof}

In the proof of Theorem \ref{immersion} we will need to understand the linear map
$$
\mathrm{ad}_{f_{n-1}}: \mathfrak{sl}(n,\C) \to \mathfrak{sl}(n,\C)
$$
in terms of the basis $\{v_{i,j}\}$.  The first thing to observe is that $\mathrm{ad}_{f_{n-1}}$ maps $S(j)$ to $S(j-(n-1))$.  In particular, if $j<0$ then $[f_{n-1},v_{i,j}]=0$.  We can also note that, since the $v_{i,j}$ are integer valued matrices, the coefficients of the matrix representing $\ad_{f_{n-1}}$ in this basis will be rational numbers.  These coefficients are known as \textit{structure constants}.  The next two lemmas describe properties of these structure constants which will be put to use in the proof of Theorem \ref{immersion}.

\

\begin{remark}
In general, given a basis $\{x_i\}$ for a complex Lie algebra $\mathfrak{a}$, the structure constants for $\mathfrak{a}$ relative to the basis $\{x_i\}$ are defined by the equation
$$
[x_i,x_j] = \sum c_{k}^{i,j}x_k.
$$
For an arbitrary basis, the structure constants $c_k^{i,j}$ are complex numbers which must satisfy certain conditions required by the Jacobi identity.

In our  case however, for the Lie algebra $\mathfrak{sl}(n,\C)$ with the basis $\{v_{i,j} | 1 \leq i \leq n-1, -i \leq j \leq i\}$ we have
$$
[v_{i,j},v_{k,\ell}]=\sum c_{m}^{i,j,k,\ell}v_{m, j + \ell}
$$
for some $c_m^{i,j,k,\ell} \in \Q$.  Below, we analyze the structure constants relevant to computing the maps $\mathrm{ad}_{\tilde{f}}=[\tilde{f},\cdot]$ and $\ad_{f_{n-1}}=[f_{n-1},\cdot]$ in the basis $\{v_{i,j}\}$. 
\end{remark}

\

\begin{lemma} \label{ftilde}
Define $a_{i,j} \in \Q$ by the condition
$$
\mathrm{ad}_{\tilde{f}}(v_{i,j}) = a_{i,j}v_{i,j-1}.
$$
Then, $a_{i,j}=(i+j)(i-j+1)$.  In particular, for $-i+1 \leq j \leq i$, we have $a_{i,j}>0$.
\end{lemma}

\begin{proof}
We proceed by induction.
First, recall that $[\tilde{e},\tilde{f}]=\tilde{h}$ and that $v_{i,j}$ is a $2j$-eigenvector for the map $\mathrm{ad}_{\tilde{h}}=[\tilde{h},\cdot]$.  Using this and the fact that $v_{i,-i} \in \ker \mathrm{ad}_{\tilde{f}}$ with the Jacobi identity, we find that
\begin{align*}
\mathrm{ad}_{\tilde{f}}(v_{i,-i+1}) &= [\tilde{f},v_{i,-i+1}] \\
&= [\tilde{f},[\tilde{e},v_{i,-i}]]\\
&= [v_{i,-i},[\tilde{e},\tilde{f}]] \\
&= -[\tilde{h},v_{i,-i}] \\
&= 2iv_{i,-i}.
\end{align*}
So, $a_{i,-i+1}=2i$.  

Now, suppose the claim holds for $j>-i+1$.  Then, we have
\begin{align*}
[\tilde{f},v_{i,j+1}] &= [\tilde{f},[\tilde{e},v_{i,j}]] \\
&= [\tilde{e},[\tilde{f},v_{i,j}]]+[v_{i,j},[\tilde{e},\tilde{f}]] \\
&= a_{i,j}[\tilde{e},v_{i,j-1}] -[\tilde{h},v_{i,j}] \\
&= (a_{i,j}-2j)v_{i,j}.
\end{align*}
This shows that $a_{i,j+1} = a_{i,j}-2j$ and so, by the inductive hypothesis, we have
$$
a_{i,j+1} = (i+j)(i-j+1)-2j = (i+j+1)(i-j).
$$
\end{proof}

The next lemma is the key fact from this section used in the proof of Theorem \ref{immersion}.

\begin{lemma} \label{ibracket}   Let $0 \leq k \leq j \leq n-1$ but $k \neq n-1$.  Write
\begin{equation} 
[f_{n-1},v_{n-1-k,n-1-j}] = \sum_{i = \max(1,j)}^{n-1} c_{i,j,k}v_{i,-j}
\end{equation}
for some rational numbers $c_{i,j,k}$.  Then for fixed $i$ and $k$, if $c_{i,k,k}\neq 0$ then for all $j$ such that $k \leq j \leq i$ we have that $c_{i,j,k}$ is non-zero and has the same sign as $c_{i,k,k}$. 
\end{lemma}

Before giving the proof, let us consider an example which will be put to use in the next section.  In the case $i=n-1$ and $k=n-2$, we can show that $c_{n-1,n-1,n-2}=2(n-1)$ and $c_{n-1,n-2,n-2}=2$ as follows.  Noting that $v_{1,0}=\tilde{h}$ and $v_{1,1} = -2\tilde{e}$, we have
\begin{align*}
[f_{n-1},v_{1,0}] & = [v_{n-1,-(n-1)},\tilde{h}] \\
&=2(n-1)v_{n-1,-(n-1)}
\end{align*}
and
\begin{align*}
[f_{n-1},v_{1,1}] &= [v_{n-1,-(n-1)},-2\tilde{e}] \\
&= 2v_{n-1,-(n-2)}.    
\end{align*}
This computation also shows that $c_{n-2,n-2,n-2}=0$ but we will not use this fact.

\begin{proof}
First, observe that $f_{n-1} \in \ker \mathrm{ad}_{\tilde{f}}$ and $\tilde{f} \in \ker \mathrm{ad}_{f_{n-1}}$.  In particular, with the Jacobi identity, this implies that $\mathrm{ad}_{f_{n-1}} \circ \mathrm{ad}_{\tilde{f}} = \mathrm{ad}_{\tilde{f}} \circ \mathrm{ad}_{f_{n-1}}$.  

Now, let $j_0=\max(1,j)$.  For $0 \leq j \leq n-2$ we have
\begin{align}
a_{n-1-k,n-1-j}[f_{n-1},v_{n-1-k,n-1-(j+1)}] &= [f_{n-1},[\tilde{f},v_{n-1-k,n-1-j}]] \nonumber \\
&= [\tilde{f},[f_{n-1},v_{n-1-k,n-1-j}]] \nonumber \\
&= \sum_{i=j_0}^{n-1}c_{i,j,k}[\tilde{f},v_{i,-j}] \nonumber \\
&= \sum_{i=j_0}^{n-1}c_{i,j,k}a_{i,-j}v_{i,-j-1}. \label{n-i}
\end{align}
Also,
\begin{equation}\label{n-i-1}
[f_{n-1},v_{n-1-k,n-1-(j+1)}] = \sum_{i=j+1}^{n-1} c_{i,j+1,k}v_{i,-j-1}
\end{equation}
Then linear independence of the $v_{i,j}$ with \eqref{n-i} and \eqref{n-i-1} together imply that
$$
c_{i,j+1,k} = \frac{a_{i,-j}}{a_{n-1-k,n-1-j}}c_{i,j,k}.
$$
It then follows by induction that
$$
c_{i,j+1,k} = \left(\prod_{\ell=k}^j \frac{a_{i,-\ell}}{a_{n-1-k,n-1-\ell}}\right)c_{i,k,k}.
$$
In particular, as $a_{i,\ell}>0$ for $1 \leq i \leq n-1$ and $-i+1 \leq \ell \leq i$, if $c_{i,k,k} \neq 0$ then $c_{i,j,k}\neq 0$ and has the same sign as $c_{i,k,k}$ for $k \leq j \leq i$.
\end{proof}

\section{Proof of Theorem \ref{immersion}} \label{proof}
As in Section \ref{oper section}, let $\mathcal{P}_{n,d}$ denote the space of (canonically normalized) meromorphic cyclic $\mathrm{SL}(n,\C)$-opers on $\C\mathbb{P}^1$ with a single pole.  An element $\nabla \in \mathcal{P}_{n,d}$ is locally equivalent to 
$$
d - \left(\begin{matrix}
0 & 1 & 0 & & & &  \\
& 0 & 1 & 0 & & & \\
& & & \ddots & & &  \\
&  & &  & 0 &  1 & 0 \\
&  & &  & & 0 & 1 \\
p & 0 & & \dots & & & 0
\end{matrix}\right)dz
$$
where $p$ is a polynomial of degree $d$.

Such a connection is gauge equivalent, using a constant gauge transformation, to one with connection form given by $\tilde{A} = \tilde{e} + cpf_{n-1}$, where $c$ is a non-zero constant.   Thus, without loss of generality, we assume $A = \tilde{e} + p f_{n-1}$.  A choice of tangent vector $\dot{A}$ amounts to a choice of polynomial $\dot{p}$ of degree strictly less than $\deg p -1$.  

With this notation the isomonodromy equation of Proposition \ref{oper jmu} becomes
\begin{equation} \label{sljmu}
\frac{d\Omega}{dz} = \dot{p}f_{n-1} + [\tilde{e},\Omega] + p[f_{n-1},\Omega]
\end{equation}
We write
$$
\Omega = \sum_{i=1}^{n-1}\sum_{j=-i}^i \omega_{i,j}v_{i,j}
$$
where the $v_{i,j}$ are as defined in section \ref{slrep}.  Then, equation \eqref{sljmu} is equivalent to a system of $n^2-1$ first order scalar differential equations.  There are four cases corresponding to the values of $j$:  
\begin{equation} \label{jpos}
\omega_{i,j}' = \omega_{i,j-1}, \quad j \geq 1;
\end{equation}
\begin{equation} \label{jneg}
\omega_{i,-j}' = \omega_{i,-j-1} + p\sum_{k=0}^{j} c_{i,j,k}\omega_{n-1-k,n-1-j}, 
\quad 0 \leq j < i \leq n-1;
\end{equation}

\begin{equation} \label{j=-i}
\omega_{i,-i}' =  p\sum_{k=0}^{i} c_{i,i,k}\omega_{n-1-k,n-1-i}, 
\quad 0 < i < n-1;
\end{equation}

\begin{equation} \label{n-1}
\omega_{n-1,-(n-1)}' = \dot{p} + p\sum_{k=0}^{n-2} c_{n-1,n-1,k}\omega_{n-1-k,0}.
\end{equation}

In the equations above the $c_{i,j,k} \in \Q$ are as defined in Lemma \ref{ibracket}.  We now argue that this system reduces to a system of $n-1$ higher order differential equations for just the $\omega_{i,i}$.

To do this we will make use of the following notation.  Given two $k$ times differentiable functions $f$ and $g$, we define the symbol $W_k(f,g)$ to be any arbitrary linear combination of derivatives of $f$ and $g$ of the form
\begin{equation} \label{Wk}
\sum_{j=0}^k c_{j}f^{(k-j)}g^{(j)}
\end{equation}
where the coefficients $c_j$ are all non-negative real numbers at least one of which is non-zero. We call $W_k(f,g)$ a \textit{weight k expression} in $f$ and $g$.  In the argument below, the weight $k$ expression $W_k(f,g)$ may denote a different linear combination each time it appears, as the coefficients $c_j$ may change. Also, note that 
\begin{equation} \label{dW}
\frac{d}{dz}W_k(f,g) = W_{k+1}(f,g)
\end{equation}
and observe that if $f$ and $g$ are non-zero polynomials, one of which is of degree $\geq k$, then 
$$
\deg W_k(f,g) = \deg f + \deg g - k.
$$  

Let us further denote by $W_k'(f,g)$ any expression of the form \eqref{Wk} with non-negative coefficients but allowing for the possibility that all are 0.  Thus, for non-zero polynomials $f$ and $g$, one of which is of degree $\geq k$, an expression $W'_k(f,g)$ is either a non-zero polynomial of degree $\deg f + \deg g -k$ or the zero polynomial.  Also, equation \eqref{dW} holds replacing $W_k$ and $W_{k+1}$ by $W'_k$ and $W'_{k+1}$, respectively.

The next proposition shows that the $n^2-1$ equations \eqref{jpos} - \eqref{n-1} reduce to $n-1$ equations given in terms of $\omega_{i,i}$ for $1 \leq i \leq n-1$ with a special form.  This proposition and the lemma to follow imply that the isomonodromy equation for meromorphic cyclic opers with a single pole has no non-trivial polynomial solutions.  The key step is an application of Lemma \ref{ibracket} which will allow us to turn a $W'$ expression into a $W$ expression.

\begin{proposition} \label{jmu}
For $1 \leq i \leq n-2$, consider the integers $j$ such that $c_{i,j,j} \neq 0$.  Let $m_i \leq i$ be the number of such integers and enumerate them in decreasing order, denoting them by  $j_{i,\ell}$ for $ 1 \leq \ell \leq m_i$ such that $i \geq j_{i,1} > j_{i,2} > \dots > j_{i,m_i} \geq 0$. 
Then, the $(2i+1)$ derivative of $\omega_{i,i}$ is a sum of weight expressions in $p$ and $\omega_{n-1-j_{i,\ell},n-1-j_{i,\ell}}$; specifically
$$
\omega_{i,i}^{(2i+1)} =  \sum_{\ell=1}^{m_i} \pm W_{i-j_{i,\ell}}(p,\omega_{n-1-j_{i,\ell},n-1-j_{i,\ell}}).
$$
For the case $i=n-1$, we have $j_{n-1,1}=n-2$ and
$$
\omega_{n-1,n-1}^{(2n-1)} = \dot{p} + W_1(p,\omega_{1,1}) + \sum_{\ell=2}^{m_{n-1}} \pm W_{n-1-j_{n-1,\ell}}(p,\omega_{n-1-j_{n-1,\ell},n-1-j_{n-1,\ell}}).
$$
\end{proposition}

\begin{proof}
From equations \eqref{jpos} and \eqref{jneg} we have
\begin{equation} \label{i+1}
\omega_{i,i}^{(i+1)} = \omega_{i,0}' = \omega_{i,-1} + c_{i,0,0}p\omega_{n-1,n-1}.
\end{equation}
We now differentiate $i-1$ more times, applying equations \eqref{jpos} and \eqref{jneg}, to give an expression for $\omega_{i,i}^{(2i)}$.  Let us give the next two steps to demonstrate how the inductive argument works.  First, differentiating \eqref{i+1} gives
\begin{align*}
\omega_{i,i}^{(i+2)} &= \omega_{i,-1}' + c_{i,0,0}(p\omega_{n-1,n-1}' + p'\omega_{n-1,n-1}) \\
&= \omega_{i,-2} + c_{i,1,1}p\omega_{n-2,n-2} + c_{i,1,0}p\omega_{n-1,n-2} + c_{i,0,0}(p\omega_{n-1,n-1}' + p'\omega_{n-1,n-1}) \\
&= \omega_{i,-2} + c_{i,1,1}p\omega_{n-2,n-2} + (c_{i,1,0} + c_{i,0,0})p\omega_{n-1,n-1}' + c_{i,0,0} p'\omega_{n-1,n-1}. 
\end{align*}
By Lemma \ref{ibracket}, $c_{i,1,0}$ and $c_{i,0,0}$ are either both 0 or both non-zero and have the same sign.  Thus, we may write
\begin{equation} \label{i+2}
\omega_{i,i}^{(i+2)}= \omega_{i,-2} + c_{i,1,1}p\omega_{n-2,n-2} \pm W'_1(p,\omega_{n-1,n-1}).
\end{equation}
The next step is similar.  Differentiating \eqref{i+2} gives
\begin{align*}
    \omega_{i,i}^{(i+3)} &= \omega'_{i,-2} + c_{i,1,1}(p'\omega_{n-2,n-2} + p \omega_{n-2,n-2}') \pm W'_2(p,\omega_{n-1,n-1}) \\
    &= \omega_{i,-3} + c_{i,2,2}p\omega_{n-3,n-3} + c_{i,2,1}p\omega_{n-2,n-3} + c_{i,2,0}p\omega_{n-1,n-3} \\
    & \quad \quad + c_{i,1,1}(p'\omega_{n-2,n-2} + p\omega'_{n-2,n-2}) \pm W'_2(p,\omega_{n-1,n-1}) \\
    &= \omega_{i,-3} + c_{i,2,2}p\omega_{n-3,n-3} + c_{i,2,1}p\omega_{n-2,n-2}' + c_{i,2,0}p\omega''_{n-1,n-1} \\
    & \quad \quad + c_{i,1,1}(p'\omega_{n-2,n-2} + p\omega'_{n-2,n-2}) \pm W'_2(p,\omega_{n-1,n-1}).
\end{align*}
The weight 2 expression $W'_2(p,\omega_{n-1,n-1})$ above is 0 unless $c_{i,0,0} \neq 0$, in which case $c_{i,2,0}$ is also non-zero and has the same sign as $c_{i,0,0}$ by Lemma \ref{ibracket}.  Also by Lemma \ref{ibracket}, $c_{i,2,1}$ and $c_{i,1,1}$ are either both 0 or both non-zero with the same sign.  
Thus, we can write
$$
\omega_{i,i}^{(i+3)} = \omega_{i,-3} + c_{i,2,2}p\omega_{n-3,n-3} \pm W'_1(p,\omega_{n-2,n-2}) \pm W'_2(p,\omega_{n-1,n-1}).
$$
Continuing in this way, it follows by induction that
$$
\omega_{i,i}^{(2i)}= \omega_{i,-i} + c_{i,i-1,i-1}p\omega_{n-i,n-i} + \sum_{k=0}^{i-2} \pm W_{i-k-1}'(p,\omega_{n-1-k,n-1-k})
$$
Differentiating once more gives
\begin{align} \label{ii}
\omega_{i,i}^{(2i+1)} = \omega_{i,-i}' & + c_{i,i-1,i-1}(p\omega_{n-i,n-i}'+p'\omega_{n-i,n-i}) \nonumber \\
&+\sum_{k=0}^{i-2} \pm W_{i-k}'(p,\omega_{n-1-k,n-1-k}).
\end{align}
It is important to note here that the expression $W'_{i-k}(p,\omega_{n-1-k,n-1-k})$ is 0 unless $c_{i,k,k} \neq 0$.

Next, from equation \eqref{j=-i} we have
\begin{equation} \label{i-i}
\omega_{i,-i}' = c_{i,i,i}p\omega_{n-1-i,n-1-i} + p\sum_{k=0}^{i-1} c_{i,i,k}\omega_{n-1-k,n-1-i}
\end{equation}
for $1 \leq i \leq n-2$.
Here, observe that for $0 \leq k \leq j \leq n-1$ but $k \neq n-1$, equation \eqref{jpos} gives
\begin{equation} \label{j to k}
 \omega_{n-1-k,n-1-k}^{(j-k)}=\omega_{n-1-k,n-1-j}.
\end{equation}
With this, equation \eqref{i-i} becomes
\begin{align}
\omega_{i,-i}'&= c_{i,i,i}p\omega_{n-1-i,n-1-i} + p\sum_{k=0}^{i-1}  c_{i,i,k}\omega_{n-1-k,n-1-k}^{(i-k)} \nonumber \\
&= c_{i,i,i}p\omega_{n-1-i,n-1-i}' + \sum_{k=0}^{i-1} \pm W'_{i-k}(p,\omega_{n-1-k,n-1-k}). \label{last step}
\end{align}
Substituting \eqref{last step} into \eqref{ii}, we have
\begin{align} \label{W'}
\omega_{i,i}^{(2i+1)} = & c_{i,i,i}p\omega_{n-1-i,n-1-i}  + (c_{i,i,i-1} + c_{i,i-1,i-1})p\omega_{n-i,n-i}' \\ &+ c_{i,i-1,i-1}p'\omega_{n-i,n-i}  
 +\sum_{k=0}^{i-2} \pm W_{i-k}'(p,\omega_{n-1-k,n-1-k}). \nonumber
\end{align} 
Again using Lemma \ref{ibracket}, we find that \eqref{W'} can be written
\begin{equation} \label{W''}
    \omega_{i,i}^{(2i+1)} = \sum_{k=0}^i \pm W'_{i-k}(p,\omega_{n-1-k,n-1-k}).
\end{equation}  
Now, as was noted above, the expression $W'_{i-k}(p,\omega_{n-1-k,n-1-k})$ appearing in \eqref{W''} is 0 unless $c_{i,k,k} \neq 0$.  As in the statement of the proposition, define a decreasing sequence $(j_{i,\ell})_{\ell=1}^{m_i}$ such that $i \geq j_{i,1} > j_{i,2} > \dots > j_{i,m_i} \geq 0$ and $c_{i,j_{i,\ell},j_{i,\ell}}\neq 0$.  Then \eqref{W''} becomes
$$
\omega_{i,i}^{(2i+1)} = \sum_{\ell=1}^{m_i} \pm W_{i-j_{i,\ell}}(p,\omega_{n-1-j_{i,\ell},n-1-j_{i,\ell}})
$$
and the proof is complete for the case $1 \leq i \leq n-2$.

For $i = n-1$, we use equation \eqref{n-1} and \eqref{j to k} to get
\begin{align*}
\omega_{n-1,-(n-1)}' &= \dot{p} + c_{n-1,n-1,n-2}p\omega_{1,0} + p\sum_{k=0}^{n-3} c_{n-1,n-1,k} \omega_{n-1-k,0} \\
& = \dot{p} + c_{n-1,n-1,n-2}p\omega_{1,1}' \pm \sum_{k=0}^{n-3} W'_{n-1-k}(p,\omega_{n-1-k,n-1-k}).
\end{align*}
Substituting this into \eqref{ii} and applying the observation that $c_{n-1,n-2,n-2}=2 \neq 0$ completes the proof.
\end{proof}

\begin{lemma} \label{no solution}
The system described in Proposition \ref{jmu} has no non-trivial polynomial solutions if $\deg \dot{p} < \deg p -1$.
\end{lemma}

\begin{proof}
Suppose polynomial solutions exist and let $d_0 = \max(\deg \omega_{i,i})$.  Then, for $1 \leq i \leq n-2$, we have
$$
\omega_{i,i}^{(2i+1)}= \sum_{\ell=1}^{m_i} \pm  W_{i-j_{i,\ell}}(p,\omega_{n-1-j_{i,\ell},n-1-j_{i,\ell}})
$$
 by Proposition \ref{jmu}.
Assuming $\deg \omega_{n-1-j_{i,1},n-1-j_{i,1}} =d_0$ we have
$$
d_0-(2i+1) \geq \deg \omega_{i,i}^{(2i+1)} = \deg p + d_0 -i + j_{i,1}.
$$
This is a contradiction.  Thus, $\deg \omega_{n-1-j_{i,1},n-1-j_{i,1}} < d_0$. 

Assume now that we have shown that $\deg \omega_{n-1-j_{i,\ell},n-1-j_{i,\ell}} < d_0$ for $1 \leq \ell \leq k-1$.  Then, if $\deg\omega_{n-1-j_{i,k},n-1-j_{i,k}} = d_0$, we have
$$
d_0-(2i+1) \geq \deg \omega_{i,i}^{(2i+1)} = \deg p + d_0 -i + j_{i,k}.
$$
Again, this is a contradiction and it follows by induction that 
$$
\deg \omega_{n-1-j_{i,\ell},n-1-j_{i,\ell}} < d_0
$$ 
for $1 \leq \ell \leq m_i$.

Finally, we deal with the case $i=n-1$.  Suppose that $\deg \omega_{1,1} = d_0$.  Then, again by Proposition \ref{jmu} and our assumption that $\deg \dot{p} < \deg p -1$, this implies that
$$
d_0 -(2n-1) > \deg \omega_{n-1,n-1}^{(2n-1)} = \deg p +d_0 -1.
$$ 
This final contradiction completes the proof.  
\end{proof}

We can now prove Theorem \ref{immersion}.  We restate it here for the reader's convenience.

\begin{theorem}
If $d = kn$ for some positive integer $k$, then the monodromy map 
$$
\nu: \mathcal{P}_{n,d} \to \mathcal{B}
$$ is a holomorphic immersion.
\end{theorem}

\begin{proof}
Let $\nabla \in \mathcal{P}_{n,d}$ and $\dot{A} \in T_\nabla\mathcal{P}_{n,d}$.
By Proposition \ref{oper jmu} and the discussion at the beginning of this section, $\dot{A} \in \ker d_\nabla\nu$ if an only if there exists a polynomial function $\Omega: \C \to \mathrm{End}(\C^n)$ satisfying \eqref{sljmu}.   But by Proposition \ref{jmu}, equation \eqref{sljmu} reduces to a system of $n-1$ equations which have no non-trivial polynomial solutions by Lemma \ref{no solution}.  Thus, no such $\Omega$ can exist.
\end{proof}

\pagebreak


\begin{thebibliography}{10}

\bibitem{acosta}
Jorge~A. Acosta.
\newblock {\em Holonomy {L}imits of {C}yclic {O}pers}.
\newblock ProQuest LLC, Ann Arbor, MI, 2016.
\newblock Thesis (Ph.D.)--Rice University.

\bibitem{bakken}
Ivar Bakken.
\newblock A multiparameter eigenvalue problem in the complex plane.
\newblock {\em Amer. J. Math.}, 99(5):1015--1044, 1977.

\bibitem{Lutz}
W.~Balser, W.~B. Jurkat, and D.~A. Lutz.
\newblock Birkhoff invariants and {S}tokes' multipliers for meromorphic linear
  differential equations.
\newblock {\em J. Math. Anal. Appl.}, 71(1):48--94, 1979.

\bibitem{boalch_thesis}
Philip Boalch.
\newblock {\em Symplectic geometry and isomonodromic deformations}.
\newblock 1999.
\newblock Thesis (Ph.D.)--Wadham College, Oxford.

\bibitem{boalch_poisson}
Philip Boalch.
\newblock Poisson varieties from {R}iemann surfaces.
\newblock {\em Indag. Math. (N.S.)}, 25(5):872--900, 2014.

\bibitem{wild_boalch}
Philip Boalch.
\newblock Wild character varieties, points on the {R}iemann sphere and
  {C}alabi's examples.
\newblock In {\em Representation theory, special functions and {P}ainlev\'{e}
  equations---{RIMS} 2015}, volume~76 of {\em Adv. Stud. Pure Math.}, pages
  67--94. Math. Soc. Japan, Tokyo, 2018.


\bibitem{ben-zvi}
Edward Frenkel and David Ben-Zvi.
\newblock {\em Vertex algebras and algebraic curves}, volume~88 of {\em
  Mathematical Surveys and Monographs}.
\newblock American Mathematical Society, Providence, RI, second edition, 2004.

\bibitem{fulton}
William Fulton and Joe Harris.
\newblock {\em Representation theory}, volume 129 of {\em Graduate Texts in
  Mathematics}.
\newblock Springer-Verlag, New York, 1991.
\newblock A first course, Readings in Mathematics.

\bibitem{jimbo}
Michio Jimbo, Tetsuji Miwa, and Kimio Ueno.
\newblock Monodromy preserving deformation of linear ordinary differential
  equations with rational coefficients. {I}. {G}eneral theory and {$\tau
  $}-function.
\newblock {\em Phys. D}, 2(2):306--352, 1981.

\bibitem{kostant}
Bertram Kostant.
\newblock The principal three-dimensional subgroup and the {B}etti numbers of a
  complex simple {L}ie group.
\newblock {\em Amer. J. Math.}, 81:973--1032, 1959.

\bibitem{nevanlinna}
Rolf Nevanlinna.
\newblock {\"U}ber riemannsche fl{\"a}chen mit endlich vielen windungspunkten.
\newblock {\em Acta Math.}, 58:295--373, 1932.

\bibitem{sibuya}
Yasutaka Sibuya.
\newblock {\em Global theory of a second order linear ordinary differential
  equation with a polynomial coefficient}.
\newblock North-Holland Publishing Co., Amsterdam-Oxford; American Elsevier
  Publishing Co., Inc., New York, 1975.
\newblock North-Holland Mathematics Studies, Vol. 18.

\bibitem{simpson}
Carlos Simpson.
\newblock The {H}odge filtration on nonabelian cohomology.
\newblock In {\em Algebraic geometry---{S}anta {C}ruz 1995}, volume~62 of {\em
  Proc. Sympos. Pure Math.}, pages 217--281. Amer. Math. Soc., Providence, RI,
  1997.

\bibitem{wasow}
Wolfgang Wasow.
\newblock {\em Asymptotic expansions for ordinary differential equations}.
\newblock Robert E. Krieger Publishing Co., Huntington, N.Y., 1976.
\newblock Reprint of the 1965 edition.

\bibitem{wentworth}
Richard~A. Wentworth.
\newblock Higgs bundles and local systems on {R}iemann surfaces.
\newblock In {\em Geometry and quantization of moduli spaces}, Adv. Courses
  Math. CRM Barcelona, pages 165--219. Birkh\"{a}user/Springer, Cham, 2016.


\end{thebibliography}
\end{document}